\documentclass[a4paper, 12pt, twoside, notitlepage]{amsart}

\usepackage{amsmath,amscd}
\usepackage{amssymb}
\usepackage{amsthm}
\usepackage{comment}
\usepackage{graphicx, xcolor}
\usepackage{mathrsfs}
\usepackage[ocgcolorlinks, linkcolor=blue]{hyperref}

\usepackage{bm}
\usepackage{bbm}
\usepackage{url}

\usepackage[utf8]{inputenc}
\usepackage{mathtools,amssymb}
\usepackage{esint}
\usepackage{tikz}
\usepackage{dsfont}
\usepackage{relsize}
\usepackage{url}
\usepackage[shortlabels]{enumitem}
\usepackage{lineno}
\usepackage{enumitem}

 \usepackage{fullpage} 
\usepackage{verbatim}
\usepackage{dsfont}
\numberwithin{equation}{section}

\allowdisplaybreaks

\graphicspath{{images/}}

\newtheorem{theorem}{Theorem}[section]
\newtheorem{lemma}[theorem]{Lemma}

\title[Normal operators for momentum ray transforms]{Normal operators for momentum ray transforms,\\ II: Saint Venant operator}
\author[S.R. Jathar]{Shubham R. Jathar}
\address{ Indian Institute of Science Education and Research (IISER) Bhopal, India}
\email {shubham18@iiserb.ac.in}

\author[M. Kar]{Manas Kar}
\address{ Indian Institute of Science Education and Research (IISER) Bhopal, India}
\email{manas@iiserb.ac.in}

\author[V. P. Krishnan]{Venkateswaran P. Krishnan}
\address{
Centre for Applicable Mathematics, Tata Institute of Fundamental Research, India }
\email{vkrishnan@tifrbng.res.in}

\author[V. A. Sharafutdinov]{Vladimir A. Sharafutdinov}
\address{Sobolev Institute of Mathematics, 4 Koptyug Av., 630090, Novosibirsk, Russia}
\email{sharaf@math.nsc.ru}

\newcommand{\I}{\mathrm{i}}
\newcommand{\rb}{\right)}
\newcommand{\lb}{\left(}

\newcommand{\R}{{\mathbb R}}

\newcommand{\PD}{\partial}

\begin{document}

\begin{abstract}
The momentum ray transform $I_m^k$ integrates a rank $m$ symmetric tensor field $f$ on ${\R}^n$ over lines with the weight $t^k$,
$I_m^kf(x,\xi)=\int_{-\infty}^\infty t^k\langle f(x+t\xi),\xi^m\rangle\,\mathrm{d}t$. Let $N^k_m=(I^k_m)^*I^k_m$ be the normal operator of $I_m^k$.
To what extent is a symmetric $m$-tensor field $f$ determined by the data $(N_m^0f,\dots,N_m^rf)$ given for some $0\le r\le m$?
The Saint Venant operator $W^r_m$ is a linear differential operator of order $m-r$ with constant coefficients on the space of symmetric $m$-tensor fields. We derive an explicit formula expressing $W^r_mf$ in terms of $(N_m^0f,\dots,N_m^rf)$. The tensor field $W^r_mf$ represents the full local information on $f$ that can be extracted from the data $(N_m^0f,\dots,N_m^rf)$.

\medskip


\end{abstract}
	\maketitle

\section{Introduction}

This article is a follow-up to our prior work \cite{JKKS1}. To ensure a self-contained presentation, we have chosen to provide only a condensed version in the introduction and Section \ref{PMR}. We refer the reader to \cite{JKKS1} for more details.

Let $f$ be a Schwartz class symmetric $m$-tensor field on $\R^n$. The $k^{\mathrm{th}}$ momentum ray transform $I_{m}^{k}f$ of $f$ is defined by
\begin{equation}
I_{m}^kf(x,\xi)=\int\limits_{\R} t^{k} f_{i_1 \cdots i_m}(x+t\xi) \xi^{i_1}\cdots \xi^{i_m}\, \mathrm{d} t\quad
\big(x\in{\mathbb R}^n, \xi\in{\mathbb R}^n,|\xi|=1,\langle x,\xi\rangle=0\big).
                                                 \label{1.1}
\end{equation}
As in \eqref{1.1}, with repeating indices, the Einstein summation convention is used throughout the article.

Let $(I^k_m)^*$ be the $L^2$ adjoint of $I^k_m$.
Instead of working directly with the momentum ray transforms, we work with the associated normal operators $N^k_m=(I^k_m)^*I^k_m$.  Being an averaging operator, $N^k_m$ represents a better measurement model than the momentum ray transforms themselves. An inversion formula was obtained in  \cite{JKKS1} which recovers a symmetric $m$-tensor $f$ from the data $(N^0_mf,\dots,N^m_mf)$. The formula is reproduced in Theorem \ref{Th3.1} below.

In this work we investigate the problem of recovering a tensor field from partial data. To what extent is a symmetric $m$-tensor field $f$ determined by the data $(N_m^0f,\dots,N_m^rf)$ given for some $0\le r\le m$?

In the next section, we will recall the definition of the Saint Venant operator
\begin{equation}
W^r_m:C^\infty({\R}^n;S^m)\to C^\infty({\R}^n;S^{m-r}\otimes S^m)\quad(0\le r\le m).
                                                 \label{1.2}
\end{equation}
It is a linear differential operator of order $m-r$ with constant coefficients.
This operator was briefly mentioned in \cite[Theorem 2.17.2]{Sharafutdinov:Book:1994}, but the operator $W=W^0_m$ was widely used throughout Chapter 2 of \cite{Sharafutdinov:Book:1994}.
It is closely related to the equation
\begin{equation}
dv=f.
                           \label{1.3}
\end{equation}
where $d=\sigma\nabla$ is the inner derivative defined in Section 2.3 below.
Namely, the equation \eqref{1.3} is solvable in a simply connected domain $U\subset{\R}^n$ if and only if the right-hand side satisfies $W^0_mf=0$, see \cite[Theorem 2.2.2]{Sharafutdinov:Book:1994}. In the case of $m=2$,  the condition $W^0_2f=0$ is popular in linear elasticity and is called {\it the deformation consistency condition}, it was obtained by Saint Venant.

For $f\in{\mathcal S}({\R}^n;S^m)$, the tensor field $W^r_mf$ represents the full {\it local} information, on the field $f$, that can be extracted from the data $(I^0_mf,\dots,I^r_mf)$, see \cite[Theorem 2.17.2]{Sharafutdinov:Book:1994}. In particular, $W^r_mf$ is uniquely determined by $(N^0_mf,\dots,N^r_mf)$. The paper \cite{Mishra:Sahoo:2021} establishes that, for $f\in \mathcal{S}\left(S^m\right)$ and for $0\leq r \leq m$, the tensor field $W^{r}_m f$ can be explicitly recovered from $(I^{0}_m f,\dots, I^{r}_mf)$. In \cite[Theorem 3.1]{Mishra:Sahoo:2023}, the kernel of the momentum ray transform is described using the Saint Venant operator. It is shown that for $f\in \mathcal{S}\left(S^m\right)$, $(I^0_mf,\dots,I^r_mf)=0$ if and only if $W^r_mf=0$. We will derive an explicit formula expressing $W^r_mf$ through $(N^0_mf,\dots,N^r_mf)$; see Theorem \ref{Th3.2} below. The latter theorem is the main result of the current work.

\section*{Acknowledgements}
SRJ and VPK would like to thank the Isaac Newton Institute for Mathematical Sciences, Cambridge, UK, for support and hospitality during \emph{Rich and Nonlinear Tomography - a multidisciplinary approach} in 2023 where part of this work was done (supported by EPSRC Grant Number EP/R014604/1).
Additionally, VPK acknowledges the support of the Department of Atomic Energy,  Government of India, under
Project No.  12-R\&D-TFR-5.01-0520.
SRJ would like to thank TIFR CAM, Bangalore, for their support and hospitality during his visit, where part of this work was conducted and acknowledges the Prime Minister's Research Fellowship (PMRF) from the Government of India for his PhD work. MK was supported by the MATRICS grant (MTR/2019/001349) of SERB.
The work of VAS was performed according to the Russian Government research assignment for IM~SB~RAS, project FWNF-2022-0006.

\section{Basic definitions and main result} \label{PMR}

\subsection{Tensor algebra}\label{PMR1}
Let $T\mathbb{R}^n=\oplus_{m=0}^\infty T^m\mathbb{R}^n$ be the complex tensor algebra over $\mathbb{R}^n$. Assuming $n$ to be fixed, the notation $T^m\mathbb{R}^n$ will be often abbreviated to $T^m$.
For a fixed orthonormal basis $(e_1,\ldots, e_n)$ of $\mathbb{R}^n$, by $u_{i_1\dots i_m}=u^{i_1\dots i_m}=u(e_{i_1},\ldots, e_{i_m})$ we denote {\it coordinates} (= {\it components}) of a tensor $u\in T^m$ with respect to the basis. There is no distinction between covariant and contravariant tensors since we use orthonormal bases only.
The standard dot product on $\mathbb{R}^n$ extends to $T^m$ by
\[
\langle u,v\rangle=u^{i_1\dots i_m}\overline{v_{i_1\dots i_m}}.
\]

Let $S^m=S^m\mathbb{R}^n$ be the subspace of $T^m$ consisting of symmetric tensors.
{\it The partial symmetrization} $\sigma(i_1\dots i_m):T^{m+k}\to T^{m+k}$ in the indices $(i_1,\dots,i_m)$ is defined by
$$
\sigma(i_1\dots i_m) u_{i_1\dots i_mj_1\dots j_k}=\frac{1}{m !} \sum_{\pi \in \Pi_m} u_{i_{\pi(1)},\dots,i_{\pi(m)}j_1\dots j_k},
$$
where the summation is performed over the group $\Pi_m$ of all permutations of the set $\{1,\dots,m\}$. In particular,
$\sigma: T^m\rightarrow S^m$ is the symmetrization in all indices.
Given $u\in S^m$ and $v\in S^k$, {\it the symmetric product} $u v\in S^{m+k}$ is defined by $uv=\sigma(u\otimes v)$. Being equipped with the symmetric product, $S^*{\R}^n=\bigoplus_{m=0}^\infty S^m{\R}^n$ becomes a commutative graded algebra that is called {\it the algebra of symmetric tensors over} ${\R}^n$.

Given $u\in S^m$, let $i_u:S^k\to S^{m+k}$ be the operator of symmetric multiplication by $u$ and let $j_u:S^{m+k}\to S^k$ be the adjoint of $i_u$. These operators are written in coordinates as
\begin{align*}
\left(i_u v\right)_{i_1 \ldots i_{m+k}} & =\sigma\left(i_1 \ldots i_{m+k}\right) u_{i_1 \ldots i_m} v_{i_{m+1} \ldots i_{m+k}} \\
\left(j_u v\right)_{i_1 \ldots i_k} & =v_{i_1 \ldots i_{m+k}} u^{i_{k+1} \ldots i_{m+k}}.
\end{align*}
For the Kronecker tensor $\delta$, the notations $i_\delta$ and $j_\delta$ will be abbreviated to $i$ and $j$ respectively.

\subsection{Tensor fields}
Recall that the Schwartz space $\mathcal{S}\left(\mathbb{R}^n\right)$ is the topological vector space consisting of $C^\infty$-smooth complex-valued functions on ${\R}^n$ that decay rapidly at infinity together with all derivatives, equipped with the standard topology.
Let $\mathcal{S}\left(\mathbb{R}^n; S^m\right)=\mathcal{S}\left(\mathbb{R}^n\right)\otimes S^m$ be the topological vector space of smooth fast decaying  symmetric $m$-tensor fields, defined on $\mathbb{R}^n$. In Cartesian coordinates, such a tensor field is written as $f=(f_{i_1\dots i_m})$ with coordinates (= components) $f_{i_1\dots i_m}=f^{i_1\dots i_m}\in\mathcal{S}\left(\mathbb{R}^n\right)$ symmetric in all indices.

We use {\it the Fourier transform}
$\mathcal{F}: \mathcal{S}(\mathbb{R}^n) \rightarrow \mathcal{S}(\mathbb{R}^n),$ $f \mapsto \widehat{f}$
in the form (hereafter \textsl{i} is the imaginary unit)
\[
\mathcal{F}{f}(y)=\frac{1}{(2 \pi)^{n / 2}} \int_{\mathbb{R}^n} e^{-\textsl{i}\langle y, x\rangle} f(x)\, \mathrm{d} x.
\]
The Fourier transform $\mathcal{F}: \mathcal{S}\left(\mathbb{R}^n; S^m\right) \rightarrow \mathcal{S}\left(\mathbb{R}^n; S^m\right)$, $f \mapsto \widehat{f}$ of symmetric tensor fields is defined component-wise:
\[
\widehat{f}_{i_1 \ldots i_m}=\widehat{f_{i_1 \ldots i_m}}.
\]
The $L^2$-product on $C_0^\infty\left({\R}^n; T^m\right)$ is defined by
\begin{equation}
(f,g)_{L^2({\R}^n; T^m)}=\int_{{\R}^n}\langle f(x),g(x)\rangle\,dx.
                               \label{2.1}
\end{equation}

\subsection{Inner derivative and divergence}
The first-order differential operator
$$
d:C^\infty({\R}^n;S^m)\to C^\infty({\R}^n;S^{m+1})
$$
defined by
$$
(df)_{i_1\dots i_{m+1}}=\sigma(i_1\dots i_{m+1})\frac{\partial f_{i_1\dots i_m}}{\partial x^{i_{m+1}}}
=\frac{1}{m+1}\Big(\frac{\partial f_{i_2\dots i_{m+1}}}{\partial x^{i_1}}+\dots
+\frac{\partial f_{i_1\dots i_m}}{\partial x^{i_{m+1}}}\Big)
$$
is called {\it the inner derivative}.

{\it The divergence}
$$
\mbox{div}:C^\infty({\R}^n;S^{m+1})\to C^\infty({\R}^n;S^m)
$$
is defined by
$$
(\mbox{div}\,f)_{i_1\dots i_m}=\delta^{jk}\,\frac{\partial f_{i_1\dots i_mj}}{\partial x^k}.
$$
The operators $d$ and $-\mbox{div}$ are formally adjoint to each other with respect to the $L^2$-product \eqref{2.1}.

\subsection{The space \texorpdfstring{${\mathcal S}(T{\mathbb S}^{n-1})$}{S(TS(n-1))}}

The Schwartz space ${\mathcal S}(E)$ is well-defined for a smooth vector bundle $E\to M$ over a compact manifold with the help of a finite atlas and partition of unity subordinate to the atlas.

In particular, the Schwartz space ${\mathcal S}(T{\mathbb S}^{n-1})$ is well defined for the tangent bundle
$$
T\mathbb S^{n-1}= \{(x, \xi) \in \mathbb{R}^n \times \mathbb{S}^{n-1}:\langle x, \xi\rangle=0\}\to\mathbb{S}^{n-1},\quad(x,\xi)\mapsto\xi
$$
of the unit sphere $\mathbb S^{n-1}=\{x\in{\R}^n:|x|=1\}$.

The Fourier transform $\mathcal{F}: \mathcal{S}\left(T \mathbb{S}^{n-1}\right) \rightarrow \mathcal{S}\left(T \mathbb{S}^{n-1}\right), \varphi \mapsto \widehat{\varphi}$ is defined by
$$
\mathcal{F}{\varphi}(y, \xi)=\frac{1}{(2 \pi)^{(n-1)/2}} \int_{\xi^{\perp}} e^{-\textsl{i}\langle y, x\rangle} \varphi(x, \xi)\,\mathrm{d} x,
$$
where $\mathrm{d} x$ is the $(n-1)$-dimensional Lebesgue measure on the hyperplane $\xi^{\perp}=\left\{x \in \mathbb{R}^n : \right.$ $\langle\xi, x\rangle=0\}$.

The $L^2$-product on ${\mathcal S}(T{\mathbb S}^{n-1})$ is defined by
\begin{equation}
(\varphi,\psi)_{L^2(T\mathbb{S}^{n-1})}=\int\limits_{\mathbb{S}^{n-1}}\int\limits_{\xi^\bot}
\varphi(x,\xi)\overline{\psi(x,\xi)}\,\mathrm{d} x\,\mathrm{d}\xi,
                              \label{2.2}
\end{equation}
where $\mathrm{d}\xi$ is the $(n-1)$-dimensional Euclidean volume form on the unit sphere $\mathbb{S}^{n-1}$.

\subsection{Momentum ray transform}
It is convenient to parameterize the family of oriented lines in ${\R}^n$ by points of the manifold $T \mathbb{S}^{n-1}$. Namely, a point
$(x,\xi)\in T \mathbb{S}^{n-1}$ determines the line $\{x+t\xi:t\in{\R}\}$ through $x$ in the direction $\xi$.

For an integer $k\ge0$, {\it the momentum ray transform}
$$
I_m^k:\mathcal{S}({\R}^n;S^m)\to\mathcal{S}\left(T \mathbb{S}^{n-1}\right)
$$
is the linear continuous operator defined by \eqref{1.1}.

\subsection{Normal operators}
The formal adjoint of the momentum ray transform $I_m^k$ with respect to $L^2$-products  \eqref{2.1} and  \eqref{2.2}
\[
\left(I_m^k\right)^*: \mathcal{S}\left(T \mathbb{S}^{n-1}\right) \rightarrow C^\infty\left(\mathbb{R}^n ; S^m\right)
\]
is expressed by
$$
\big((I_m^k)^*\varphi\big)_{i_1 \ldots i_m}(x)
=\int_{\mathbb{S}^{n-1}}\langle x, \xi\rangle^k \xi_{i_1} \dots \xi_{i_m} \varphi\big(x-\langle x, \xi\rangle \xi, \xi\big)\, \mathrm{d}\xi.
$$
We emphasize that, for $\varphi \in \mathcal{S}(T \mathbb{S}^{n-1})$, the tensor field $(I^k_m)^*\varphi$ does not need to fast decay at infinity.

Let
$$
N_m^k=(I_m^k)^*I_m^k: \mathcal S\left(\mathbb{R}^n ; S^m\right) \rightarrow C^\infty\left(\mathbb{R}^n ; S^m\right)
$$
be the normal operator for the momentum ray transform $I_m^k$.
For $f\in\mathcal S\left(\mathbb{R}^n ; S^m\right)$, the Fourier transform
$\widehat{N_m^kf}\in{\mathcal S}'\left(\mathbb{R}^n ; S^m\right)$ is well defined at least in the distribution sense and the restriction of $\widehat{N_m^kf}$ to ${\R}^n\setminus\{0\}$ belongs to
$C^\infty\left(\mathbb{R}^n\setminus\{0\}; S^m\right)$.

\subsection{The inversion formula}
Let $\Gamma$ be Euler's Gamma function and let the operator $(-\Delta)^{1/2}$ be defined with the help of the Fourier transform by $|y|{\mathcal F}={\mathcal F}(-\Delta)^{1/2}$.
We use the definition
$$
(2l+1)!!=1\cdot3\cdots(2l+1),\quad (-1)!!=1.
$$
Let us reproduce \cite[Theorem 3.1]{JKKS1}.

\begin{theorem} \label{Th3.1}
Given integers $m\ge0$ and $n\ge2$,
 a tensor field $f\in \mathcal{S}\left(\mathbb{R}^n; S^m\right)$ is recovered from the data $(N_m^0f,N_m^1f,\dots,N_m^mf)$
by the inversion formula
\begin{equation}
f(x)=(-\Delta)^{1/2}\sum\limits_{k=0}^m D_{m,n}^k(N_m^kf)(x),
                           \label{2.3}
\end{equation}
where the linear differential operator of order $m+k$
$$
D_{m,n}^k:C^\infty({\R}^n;S^m)\to C^\infty({\R}^n;S^m)
$$
is defined by
\begin{align}\label{2.4}
    \begin{split}
        D_{m,n}^k&=c_{m,n}^k\sum\limits_{p=k}^m(n\!+\!2m\!-\!2p\!-\!3)!!
\!\!\!\\&\qquad\qquad\times\sum\limits_{q=0}^{\min(p,m-p,p-k)}\!\!\!\!
\frac{(-1)^q}{2^q q! (m\!-\!p\!-\!q)! (p\!-\!k\!-\!q)!}
\, d^{p-q}\, i^q\, j^q\, j_x^{p-k-q}\, \mbox{\rm div}^k
    \end{split}
\end{align}
with the coefficient
\begin{equation}
c_{m,n}^k=\frac{(-1)^k}{(k!)^2}\,\frac{2^{m-2}\, \Gamma\big(\frac{2m+n-1}{2}\big)}{\pi^{(n+1)/2}\, (n+2m-3)!!}
\end{equation}
and the operators \(i\), \(j\), and \(j_x\) are defined in Section \ref{PMR1}.
\end{theorem}

\subsection{The Saint Venant operator}
For integers $m$ and $r$ satisfying $0\le r\le m$, let $S^{m-r}\otimes S^m$ be the space of $(2m-r)$-tensors on ${\R}^n$ which are symmetric
in first $m-r$ and last $m$ indices. {\it The  Saint Venant operator} \eqref{1.2}
is defined by
\begin{equation}
\begin{aligned}
\lb  W^r_m f\rb_{i_1\dots i_{m-r}j_1 \dots j_m}=\sigma(i_1 \dots i_{m-r})&\sigma(j_1 \cdots j_m)\sum\limits_{l=0}^{m-r}(-1)^l\binom{m-r}{l}
\\
&\times\frac{\PD^{m-r}f_{i_1\dots i_{m-r-l}j_1\dots j_{r+l}}}{\PD x_{i_{m-r-l+1}}\dots \PD x_{i_{m-r}}\PD x_{j_{r+l+1}}\dots \PD x_{j_{m}} }.
\end{aligned}
                           \label{2.5}
\end{equation}
In particular $W^m_m$ is the identity operator.

\subsection{The main result}
\begin{theorem}\label{Th3.2}
Let $0\le r\le m$ and $n\ge2$ be integers. For $f\in{\mathcal S}({\R}^n;S^m)$, the tensor field $W^r_mf$ is recovered from the data $(N^0_mf,\dots,N^r_mf)$ by the inversion formula
$$
W^r_m f=(-\Delta)^{1/2}W^r_m\sum\limits_{k=0}^r  D_{m,n}^k(N^k_mf),
$$
where the linear differential operator $D_{m,n}^{k}$ is defined by \eqref{2.4}.
\end{theorem}

Theorem \ref{Th3.2} is a generalization of Theorem \ref{Th3.1} since $W^m_m$ is the identity operator.
In the case of $r=0$ Theorem \ref{Th3.2} actually coincides with \cite[Theorem 2.12.3]{Sharafutdinov:Book:1994}.

The first step in the proof of Theorem \ref{Th3.2} is as follows. Since $W^r_m$ is a differential operator with constant coefficients, it commutes with $(-\Delta)^{1/2}$.
Applying the operator $W^r_m$ to the equality \eqref{2.3}, we write the result in the form
$$
W^r_mf=(-\Delta)^{1/2}W^r_m\sum\limits_{k=0}^r  D_{m,n}^k(N^k_mf)+(-\Delta)^{1/2}W^r_m\sum\limits_{k=r+1}^m  D_{m,n}^k(N^k_mf).
$$
Thus, to prove Theorem \ref{Th3.2}, it suffices to demonstrate that
\begin{equation}
W^r_m D_{m,n}^k=0\quad\mbox{for}\quad 0\le r<k\le m.
                           \label{2.6}
\end{equation}
The proof of \eqref{2.6} is presented in the next section. 

\section{Proof of Theorem \ref{Th3.2}}\label{sec:main:proof}

Applying the Fourier transform to \eqref{2.5}, we obtain
$$
\widehat{ W^r_m f}={\textsl i}^{m-r}\,{\widehat W}^r_m\widehat f,
$$
where ${\textsl i}$ is the imaginary unit and the purely algebraic operator
$$
{\widehat W}^r_m={\widehat W}^r_m(y):S^m\to S^{m-r}\otimes S^m\quad(y\in{\R}^n)
$$
is defined by
$$
\begin{aligned}
({\widehat W}^r_mh)_{i_1\dots i_{m-r}j_1 \dots j_m}&=\sigma(i_1 \dots i_{m-r})\sigma(j_1 \cdots j_m)\sum\limits_{l=0}^{m-r}(-1)^l
\binom{m\!-\!r}{l}\times\\
&\quad\times h_{i_1\dots i_{m-r-l}j_1\dots j_{r+l}}\,y_{i_{m-r-l+1}}\dots y_{i_{m-r}} y_{j_{r+l+1}}\dots y_{j_{m}}.
\end{aligned}
$$
This can be written in the coordinate-free form
\begin{equation}
\langle{\widehat W}^r_mh,u\otimes v\rangle=
\sum\limits_{l=0}^{m-r}(-1)^l
\binom{m\!-\!r}{l}\langle h,(j_y^lu)(j_y^{m-r-l}v)\rangle
\quad\mbox{for}\ u\in S^{m-r}\ \mbox{and}\ v\in S^m.
                           \label{3.1}
\end{equation}

On the other hand, applying the Fourier transform to \eqref{2.6}, we see that \eqref{2.6} is equivalent to the statement
\begin{equation}
{\widehat W}^r_m {\widehat D}_{m,n}^k=0\quad\mbox{for}\quad 0\le r<k\le m,
                           \label{3.2}
\end{equation}
where the operator ${\widehat D}_{m,n}^k$ is defined by
\begin{align}  \label{3.3}
    \begin{split}
        \widehat D^k_{m,n}&=c^k_{m,n}\sum\limits_{p=k}^m(-1)^p(n\!+\!2m\!-\!2p\!-\!3)!! \!\!\!\!\\&\qquad\qquad\times\sum\limits_{q=0}^{\min(p,m-p,p-k)}\!\!\!\!\!\!\!\!\!
\frac{1}{2^q\,q!(m\!-\!p\!-\!q)!(p\!-\!k\!-\!q)!}\,
i_y^{p-q}i^qj^q\,\mbox{\rm div}^{p-k-q}\,j_y^k,
    \end{split}
\end{align}
see \cite[formula (8.7)]{JKKS1}.

We will use only one property of the operator ${\widehat D}_{m,n}^k$: as is seen from \eqref{3.3},
\begin{equation}
    {\widehat D}_{m,n}^k=i_y^{r+1}\,B_{m,n}^k,\qquad\quad \text{for } 0\le r<k,
\end{equation}
with some linear operator $B_{m,n}^k$.
Therefore, to prove \eqref{3.2}, it suffices to demonstrate that
\begin{equation}
{\widehat W}^r_m i_y^{r+1}=0\quad\mbox{for}\quad 0\le r\le m-1.
                           \label{3.4}
\end{equation}
By \eqref{3.1},
$$
\begin{aligned}
\langle{\widehat W}^r_m i_y^{r+1} h,u\otimes v\rangle&=
\sum\limits_{l=0}^{m-r}(-1)^l \binom{m\!-\!r}{l}
\langle i_y^{r+1}h,(j_y^lu)(j_y^{m-r-l}v)\rangle\\
&=\Big\langle h,\sum\limits_{l=0}^{m-r}(-1)^l
\binom{m\!-\!r}{l}\, j_y^{r+1}\big((j_y^lu)(j_y^{m-r-l}v)\big)\Big\rangle.
\end{aligned}
$$
This means that \eqref{3.4} holds for any $h\in S^{m-1}$ if and only if
\begin{equation}
\sum\limits_{l=0}^{m-r}(-1)^l
\binom{m\!-\!r}{l}\, j_y^{r+1}\big((j_y^lu)(j_y^{m-r-l}v)\big)=0
\quad\mbox{for any}\ u\in S^{m-r}\ \mbox{and}\ v\in S^m\ (0\le r<m).
                           \label{3.5}
\end{equation}
The left-hand side of \eqref{3.5} is homogeneous of degree $m+1$ in $y$. It suffices to prove \eqref{3.5} for a unit vector $y$. In what follows, $y\in{\R}^n$ is a fixed vector satisfying $|y|=1$.

The complex vector space $S^m=S^m{\R}^n$ is generated by powers $x^m\ (x\in{\R}^n)$. Therefore \eqref{3.5} is equivalent to the statement
$$
\sum\limits_{l=0}^{m-r}(-1)^l
\binom{m\!-\!r}{l}\, j_y^{r+1}\big((j_y^lx^{m-r})(j_y^{m-r-l}z^m)\big)=0
\quad\mbox{for any}\ x,z\in {\R}^n\quad (0\le r<m).
$$
Since $j_y^lx^{m-r}=\langle x,y\rangle^l x^{m-r-l}$ and $j_y^{m-r-l}z^m=\langle z,y\rangle^{m-r-l} z^{r+l}$, the latter statement can be written as
\begin{equation}
\sum\limits_{l=0}^{m-r}(-1)^l
\binom{m\!-\!r}{l}\,\langle x,y\rangle^l \langle z,y\rangle^{m-r-l} j_y^{r+1}(x^{m-r-l}z^{r+l})=0
                           \label{3.6}
\end{equation}
for any $x,z\in {\R}^n$ and $0\le r<m$.
The equality \eqref{3.6} holds in the case $\langle x,y\rangle=\langle z,y\rangle=0$ since all summands on the left-hand side are equal to zero.

Next, we prove \eqref{3.6} in the case $\langle x,y\rangle=0$ but $\langle z,y\rangle\neq0$. In this case \eqref{3.6} looks as follows:
\begin{equation}
j_y^{r+1}(x^{m-r}z^r)=0.
                           \label{3.7}
\end{equation}
Let us write \eqref{3.7} in coordinates
$$
y^{i_1}\dots y^{i_{r+1}}\sum\limits_{\pi\in\Pi_m}x_{i_{\pi(1)}}\dots x_{i_{\pi(m-r)}}z_{i_{\pi(m-r+1)}}\dots z_{i_{\pi(m)}}=0.
$$
After pulling the factor $y^{i_1}\dots y^{i_{r+1}}$ inside the sum, every summand contain at least one factor of the form $y^kx_k=0$. This proves \eqref{3.7}.

Quite similarly \eqref{3.6} is proved in the case $\langle x,y\rangle\neq0$ but $\langle z,y\rangle=0$.

Now, we prove \eqref{3.6} in the general case when $\alpha=\langle x,y\rangle\neq0$ and $\beta=\langle z,y\rangle\neq0$.
We represent vectors $x,z\in{\R}^n$ in the form
$$
x=\alpha y+x',\ \langle x',y\rangle=0;\quad
z=\beta y+z',\ \langle z',y\rangle=0.
$$
From this
$$
\begin{aligned}
x^{m-r-l}z^{r+l}&=(\alpha y+x')^{m-r-l}(\beta y+z')^{r+l}\\
&=\sum\limits_{p=0}^{m-r-l}\sum\limits_{q=0}^{r+l}\binom{m\!-\!r-l}{p}\binom{\!r+l}{q}
\alpha^{m-r-l-p}\beta^{r+l-q}\,y^{m-p-q}x'{}^pz'{}^q.
\end{aligned}
$$
Substituting this expression into \eqref{3.6}, we obtain (up to a factor $\alpha^{m-r} \beta^{m}$)
$$
\sum\limits_{l=0}^{m-r}\sum\limits_{p=0}^{m-r-l}\sum\limits_{q=0}^{r+l}(-1)^l
\binom{m\!-\!r}{l}\binom{m\!-\!r\!-\!l}{p}\binom{r\!+\!l}{q}\,
\alpha^{-p}\beta^{-q}\,j_y^{r+1}(y^{m-p-q}x'{}^pz'{}^q)=0.
$$
Denoting $\tilde x=\alpha^{-1}x'$ and $\tilde z=\beta^{-1}z'$, this can be written in the form
$$
\sum\limits_{l=0}^{m-r}\sum\limits_{p=0}^{m-r-l}\sum\limits_{q=0}^{r+l}(-1)^l
\binom{m\!-\!r}{l}\binom{m\!-\!r\!-\!l}{p}\binom{r\!+\!l}{q}\,
j_y^{r+1}(y^{m-p-q}\tilde x{}^p\tilde z{}^q)=0.
$$
To simplify notations, we denote $\tilde x$ and $\tilde z$ again by $x$ and $z$ respectively. Thus, we have to prove the statement
\begin{equation}
\sum\limits_{l=0}^{m-r}\sum\limits_{p=0}^{m-r-l}\sum\limits_{q=0}^{r+l}(-1)^l
\binom{m\!-\!r}{l}\binom{m\!-\!r\!-\!l}{p}\binom{r\!+\!l}{q}\,
j_y^{r+1}(y^{m-p-q}x^pz^q)=0
                           \label{3.8}
\end{equation}
for $x,z\in y^\bot$ and $0\le r<m$.

Since the last factor $j_y^{r+1}(y^{m-p-q}x^pz^q)$ on the left-hand side of \eqref{3.8} is independent of $l$, it makes sense to change the order of summations. We first change the order of summations over $l$ and $p$
$$
\sum\limits_{p=0}^{m-r}\sum\limits_{l=0}^{m-r-p}\sum\limits_{q=0}^{r+l}(-1)^l
\binom{m\!-\!r}{l}\binom{m\!-\!r\!-\!l}{p}\binom{r\!+\!l}{q}\,
j_y^{r+1}(y^{m-p-q}x^pz^q)=0
$$
and then change the order of summations over $l$ and $q$
$$
\sum\limits_{p=0}^{m-r}\sum\limits_{q=0}^{m-p}\sum\limits_{l=\max(0,q-r)}^{m-r-p}(-1)^l
\binom{m\!-\!r}{l}\binom{m\!-\!r\!-\!l}{p}\binom{r\!+\!l}{q}\,
j_y^{r+1}(y^{m-p-q}x^pz^q)=0
$$
This can be written in the form
\begin{equation}
\sum\limits_{p=0}^{m-r}\sum\limits_{q=0}^{m-p}
C(m,r,p,q)\,j_y^{r+1}(y^{m-p-q}x^pz^q)=0
\quad(x,z\in y^\bot,0\le r<m),
                           \label{3.9}
\end{equation}
where
\begin{align} \label{3.10}
    \begin{split}
     & C(m,r,p,q)\\&\quad=\!\!\!\sum\limits_{l=\max(0,q-r)}^{m-r-p}\!\!\!(-1)^l
\binom{m\!-\!r}{l}\binom{m\!-\!r\!-\!l}{p}\binom{r\!+\!l}{q}
\quad(0\le p\le m-r, 0\le q\le m-p).  
    \end{split}
\end{align}
From \eqref{3.9} and \eqref{3.10}, for $x,z\in y^\bot$, we have
\begin{equation}
j_y^{r+1}(y^{m-p-q}x^pz^q)=0\quad\mbox{\rm if}\quad p\ge0,q\ge0,p+q\le m,r+1> m-p-q.
                           \label{3.11}
\end{equation}
 Indeed, writing in coordinates
\begin{align*}
   & (y^{m-p-q}x^pz^q)_{i_1\dots i_m}
\\ &\qquad\qquad=\frac{1}{m!}\sum\limits_{\pi\in\Pi_m}
y_{i_{\pi(1)}}\dots y_{i_{\pi(m-p-q)}}
x_{i_{\pi(m-p-q+1)}}\dots x_{i_{\pi(m-q)}}
z_{i_{\pi(m-q+1)}}\dots z_{i_{\pi(m)}},
\end{align*}
we have
$$
\begin{aligned}
&\big(j_y^{r+1}(y^{m-p-q}x^pz^q)\big)_{i_{m-r}\dots i_m}\\
&=\frac{1}{m!}\sum\limits_{\pi\in\Pi_m}y^{i_1}\dots y^{i_{r+1}}\,
y_{i_{\pi(1)}}\dots y_{i_{\pi(m-p-q)}}
x_{i_{\pi(m-p-q+1)}}\dots x_{i_{\pi(m-q)}}
z_{i_{\pi(m-q+1)}}\dots z_{i_{\pi(m)}}.
\end{aligned}
$$
In the case of $r+1>m-p-q$, every summand of the sum contains either a factor of the form $y^jx_j=0$ or a factor of the form $y^jz_j=0$.

In virtue of \eqref{3.11}, the summation in \eqref{3.9} can be restricted to $(p,q)$ satisfying
\begin{equation}
p\ge0,\quad q\ge0,\quad p+q\le m-r-1.
                           \label{3.12}
\end{equation}
In particular, $r<m$ and $p\le m-r-1$.
In other words, \eqref{3.9} is equivalent to the statement
\begin{equation}
\sum\limits_{p=0}^{m-r-1}\sum\limits_{q=0}^{m-r-p-1}
C(m,r,p,q)\,j_y^{r+1}(y^{m-p-q}x^pz^q)=0
\quad(x,z\in y^\bot,0\le r<m).
                           \label{3.13}
\end{equation}

\begin{lemma} \label{L9.1}
For integers $m,r,p,q$ satisfying \eqref{3.12} and $0\le r<m$, the following equality holds:
\begin{equation}
\sum\limits_{l=\max(0,q-r)}^{m-r-p}(-1)^l
\binom{m\!-\!r}{l}\binom{m\!-\!r\!-\!l}{p}\binom{r\!+\!l}{q}=0.
                           \label{3.14}
\end{equation}
\end{lemma}

With the help of Lemma \ref{L9.1}, we immediately complete the proof of Theorem \ref{Th3.2}. Indeed, by comparing \eqref{3.10} and \eqref{3.14}, we observe that all coefficients $C(m,r,p,q)$ participating in \eqref{3.13} are equal to zero. This proves \eqref{3.9}. As shown earlier, \eqref{3.9} implies the statement of Theorem \ref{Th3.2}.

\begin{proof}[Proof of Lemma \ref{L9.1}]
We assume binomial coefficients $\binom{k}{p}$ to be defined for all integers $k$ and $p$ under the agreement
$$
\binom{k}{p}=0\quad \mbox{if either}\  k<0\ \mbox{or}\ p<0\ \mbox{or}\ k<p.
$$
Then
\begin{equation}
\begin{aligned}
C(m,r,p,q)&=\sum\limits_{l=\max(0,q-r)}^{m-r-p}(-1)^l
\binom{m\!-\!r}{l}\binom{m\!-\!r\!-\!l}{p}\binom{r\!+\!l}{q}\\
&=\sum\limits_{l=-\infty}^\infty(-1)^l
\binom{m\!-\!r}{l}\binom{r\!+\!l}{q}\binom{m\!-\!r\!-\!l}{p}.
\end{aligned}
                           \label{3.15}
\end{equation}
From \cite[p. 10]{Egorychev:1984}, we have for $0<\varepsilon\ll 1$,
$$
\binom{n}{k}=\frac{1}{2 \pi \I} \int\limits_{|z|=\varepsilon} \frac{(1+z)^n}{z^{k+1}}\, d z.
$$
In particular,
$$
\binom{r+l}{q}=\frac{1}{2\pi \I}\int\limits_{|z|=\epsilon}\frac{(1+z)^{r+l}}{z^{q+1}}\,dz,\quad
\binom{m-r-l}{p}=\frac{1}{2\pi \I}\int\limits_{|w|=\epsilon}\frac{(1+w)^{m-r-l}}{w^{p+1}}\,dw.
$$
With the help of these formulas, we transform \eqref{3.15} as follows:
$$
\begin{aligned}
C(m,r,p,q)&
=-\frac{1}{(2\pi)^2}\int\limits_{|z|=\epsilon}\int\limits_{|w|=\epsilon}
\frac{(1+z)^r(1+w)^{m-r}}{z^{q+1}w^{p+1}}
\sum\limits_{l=-\infty}^\infty(-1)^l
\binom{m\!-\!r}{l}\Big(\frac{1+z}{1+w}\Big)^l\,dw\,dz\\
&=-\frac{1}{(2\pi)^2}\int\limits_{|z|=\epsilon}\int\limits_{|w|=\epsilon}
\frac{(1+z)^r(1+w)^{m-r}}{z^{q+1}w^{p+1}}
\Big(1-\frac{1+z}{1+w}\Big)^{m-r}\,dw\,dz\\
&=-\frac{1}{(2\pi)^2}\int\limits_{|z|=\epsilon}\int\limits_{|w|=\epsilon}
\frac{(1+z)^r(w-z)^{m-r}}{z^{q+1}w^{p+1}}\,dw\,dz\\
&=-\frac{1}{(2\pi)^2}\int\limits_{|z|=\epsilon}\int\limits_{|w|=\epsilon}
\frac{(1+z)^r}{z^{q+1}w^{p+1}}
\sum\limits_{l=-\infty}^\infty(-1)^l
\binom{m\!-\!r}{l}z^lw^{m-r-l}\,dw\,dz.
\end{aligned}
$$
We perform the integration with respect to $w$. By the Cauchy integral formula, the only summand that survives corresponds to $l=m-r-p$. Thus,
$$
C(m,r,p,q)=\frac{(-1)^{m-r-p}}{2\pi{\mathrm i}}\binom{m\!-\!r}{p}\int\limits_{|z|=\epsilon}
(1+z)^rz^{m-r-p-q-1}\,dz.
$$
The integrand is a holomorphic function if $p+q\le m-r-1$. Therefore, $C(m,r,p,q)=0$  if $p+q\le m-r-1$.
\end{proof}

\bibliography{math}

\bibliographystyle{alpha}
\end{document}